\documentclass[10pt]{article}

\usepackage{epsfig}
\usepackage{amssymb, latexsym}
\usepackage{amscd}
\usepackage[all]{xy}
\usepackage{epsf}
\usepackage[mathscr]{eucal}
\usepackage{amsmath,amsfonts,amscd,amssymb,amsthm}
\usepackage{appendix}

\long\def\comment#1\endcomment{}

\theoremstyle{plain}

\newtheorem{theorem}{{\sc Theorem}}[section]
\newtheorem{lemma}[theorem]{\sc Lemma}

\newtheorem{coroll}[theorem]{\sc Corollary}

\newcommand{\appsection}[1]{\let\oldthesection\thesection
  \renewcommand{\thesection}{\sc{Appendix \oldthesection}}
  \section{#1}\let\thesection\oldthesection}

\newcommand{\appsubsection}[1]{\let\oldthesubsection\thesubsection
  \renewcommand{\thesubsection}{\sc\oldthesubsection}
  \subsection{#1}\let\thesubsection\oldthesubsection}

\newcommand{\apptheorem}{
\let\oldthetheorem\thetheorem
\renewcommand{\thetheorem}{\sc \oldthetheorem}
\theorem\let\thetheorem\oldthetheorem}

\renewcommand{\thesection}{\sc \arabic {section}}
\renewcommand{\thetheorem}{\thesection.{\sc \arabic {theorem}}}

\theoremstyle{plain}

\newtheorem{defn}[theorem]{\sc Definition}

\theoremstyle{exercise}
\newtheorem{remark}[theorem]{\sc Remark}

\newtheorem{example}[theorem]{\sc Example}

\makeatletter \@addtoreset{equation}{section} \makeatother

\def\eqref#1{\thetag{\ref{#1}}}

\let\latexref=\ref
\def\ref#1{{\normalfont{\latexref{#1}}}}

\newcommand{\udot}{{\:\raisebox{3pt}{\text{\circle*{1.5}}}}}
\def\dlim_#1{{\displaystyle\lim_{#1}}^\hdot}

\newcommand{\gr}{\mathrm{gr}}

\newcommand{\Ext}{\operatorname{Ext}}

\newcommand{\Ker}{\operatorname{{\rm Ker}}}

\newcommand{\eqto}{\mathrel{\stackrel{\sim}{\to}}}

\newcommand{\Mod}{\mathscr{M}\mathrm{od}}

\newcommand{\Cobar}{\mathrm{Cobar}}

\newcommand{\free}{\mathrm{free}}

\newcommand{\Alt}{\mathrm{Alt}}

\newcommand{\Cycl}{\mathrm{Cycl}}
\newcommand{\Symm}{\mathrm{Symm}}

\newcommand{\g}{\mathfrak{g}}

\renewcommand{\Bar}{\mathrm{Bar}}
\newcommand{\pol}{\mathrm{pol}}

\title{\sc The PBW property for associative algebras as an integrability condition}
\author{\sc Boris Shoikhet}
\date{}

\begin{document}\maketitle

\begin{abstract}
We develop an elementary method for proving the PBW theorem for associative algebras with an ascending filtration.
The idea is roughly the following. At first, we deduce a proof of the PBW property for the {\it ascending} filtration (with the filtered degree equal to the total degree in $x_i$'s) to a suitable PBW-like property for the {\it descending} filtration (with the filtered degree equal to the power of a polynomial parameter $\hbar$, introduced to the problem). This PBW property for the descending filtration guarantees the genuine PBW property for the ascending filtration, for almost all specializations of the parameter $\hbar$.
At second, we develop some very constructive method for proving this PBW-like property for the descending filtration by powers of $\hbar$, emphasizing its integrability nature.

We show how the method works in three examples.
As a first example, we give a proof of the classical Poincar\'{e}-Birkhoff-Witt theorem for Lie algebras. As a second, much less trivial example, we present a new proof of a result of Etingof and Ginzburg [EG] on PBW property of algebras with a cyclic non-commutative potential in three variables.
Finally, as a third example, we found a criterium, for a general quadratic algebra which is the quotient-algebra of $T(V)[\hbar]$ by the two-sided ideal, generated by $(x_i\otimes x_j-x_j\otimes x_i-\hbar\phi_{ij})_{i,j}$, with $\phi_{ij}$ general quadratic non-commutative polynomials, to be a PBW for generic specialization $\hbar=a$. This result seems to be new.
\end{abstract}

\section*{\sc Introduction}
\subsection{}
Let $\mathfrak{g}$ be a Lie algebra over a field of characteristic 0.
Recall that its universal enveloping algebra is defined as the quotient-algebra
$$
\mathcal{U}(\mathfrak{g})=T(\mathfrak{g})/(x\otimes y-y\otimes x-[x,y])_{x,y\in\mathfrak{g}}
$$
of the tensor algebra by the two-sided ideal generated by expressions $x\otimes y-y\otimes x-[x,y]$, for all $x,y\in\mathfrak{g}$. The universal enveloping algebra admits a natural ascending filtration $\{F_k\}$, where $F_k\subset \mathcal{U}(\mathfrak{g})$ is a $k$-subspace generated by all monomials of degree $\le k$ in generators. It is an algebra filtration,
$$
F_i*F_j\subset F_{i+j}
$$
and one defines the associated graded algebra
$$
\gr_F\mathcal{U}(\mathfrak{g})=\oplus_{i\ge -1}F_{i+1}/F_{i}
$$
where $F_{-1}=0$.

The classical Poincar\'{e}-Birkhoff-Witt theorem says that the algebras $\gr_F\mathcal{U}(\mathfrak{g})$ and the symmetric algebra $S(\mathfrak{g})$ are (canonically) isomorphic as graded algebras. Essentially it means that
$$
F_{k}/F_{k-1}\simeq S^k(\mathfrak{g})
$$

Note that this PBW theorem can be thought as a consequence of a kind of integrability condition. For example, if an associative algebra $A$ with generators $x_1,\dots,x_n$ is defined by relations
$$
x_i\otimes x_j-x_j\otimes x_i-\sum_kc_{ij}^kx_k=0
$$
for each pair $1\le i<j\le n$,
but the ``structure constants'' $c_{ij}^k$ do not obey the Jacobi identity, the Poincar\'{e}-Birkhoff-Witt theorem fails.
The corresponding quotient-algebra by the two-sided ideal has a ``smaller size'' than $S(x_1,\dots,x_n)$.

\subsection{}
More generally, consider a vector space $V$ over $k$ with basis $\{x_1,\dots,x_n\}$, consider the tensor algebra $T(V)$, and its $\hbar$-linear version $T(V)[\hbar]$, where $\hbar$ is a formal parameter. For each pair $1\le i<j\le n$ consider an element $\varphi_{ij}\in\hbar\cdot T(V)[\hbar]$, and consider the quotient-algebra
\begin{equation}\label{eq01}
A=T(V)[\hbar]/(x_i\otimes x_j-x_j\otimes x_i-\varphi_{ij})_{1\le i<j\le n}
\end{equation}
When $\deg_{x}\varphi_{ij}\le 2$ for all $i<j$, the algebra $A$ admits the ascending (Poincar\'{e}-Birkhoff-Witt-like) filtration $\{F_k\}$ where $F_k$ is spanned by the monomials of total degree $\le k$ in $\{x_i\}$. The following problem raises up naturally:

{\it Whether it is true, for a particular algebra $A$ of type \eqref{eq01}, that for any $k\ge 0$ one has $F_k/F_{k+1}=S^k(x_1,\dots,x_n)[\hbar]$, for any $k\ge 0$?}

We call this property {\it the Poincar\'{e}-Birkhoff-Witt, or shortly PBW, property}.

In practice, it may be very difficult to check our whether an algebra $A$ of type \eqref{eq01} is a PBW algebra. In this paper, we suggest a very general method for proving the PBW property, which emphasizes the ``integrability'' nature of the PBW condition.

An algebra $A$ of type \eqref{eq01} admits, along with the ascending filtration $\{F_k\}$, the following descending filtration $\{\Gamma_k\}$:
\begin{equation}
\Gamma_k=\hbar^k\cdot A
\end{equation}
It is an algebra filtration as well, and the associated graded algebra is defined:
\begin{equation}
\gr_{\Gamma}A=\oplus_{i\ge 0}\Gamma_i/\Gamma_{i+1}
\end{equation}

Our first Theorem \eqref{main1} shows that, provided some PBW-like property for the descending filtration $\Gamma$ is fulfilled, the algebra $A$ is a PBW algebra, for all but a countable number of specializations at $a\in k$.

This PBW-like property for the descending filtration $\Gamma$ reads:
\begin{equation}\label{eq05}
\text{For each $k\ge 0$, one has $\Gamma_k/\Gamma_{k+1}=\hbar^k\cdot S(V)$}
\end{equation}

On the other side, we develop a very powerful method for proving \eqref{eq05}, which hopefully may work in many cases, in our second Theorem \ref{lemma21bis}.

It is instructive to consider what happens with $\gr_\Gamma A$ for our ``not-a-Lie-algebra'' example, that is, for the relations
\begin{equation}
x_i\otimes x_j-x_j\otimes x_i-\hbar\sum_kc_{ij}^kx_k
\end{equation}
where $\{c_{ij}^k\}$ fail to satisfy the Jacobi identity.

It is easy to see that $\Gamma_0/\Gamma_1\simeq S(V)$ and $\Gamma_1/\Gamma_2\simeq \hbar\cdot S(V)$.
However, suppose $\{c_{ij}^k\}$ do not satisfy the Jacobi identity:
\begin{equation}
\sum_{a}\left(c_{ij}^ac_{ak}^b+c_{jk}ac_{ai}^b+c_{ki}^ac_{aj}^b\right)\ne 0
\end{equation}
for some $i,j,k$ and for some $b$.

An easy computation shows that
\begin{equation}\label{eq09}
[[x_i,x_j],x_k]+[[x_j,x_k],x_i]+[[x_k,x_i],x_j]=\hbar^2\sum_{a,b}(c_{ij}^ac_{ak}^b+c_{jk}ac_{ai}^b+c_{ki}^ac_{aj}^b)\cdot x_b
\end{equation}
where $[x,y]:=x*y-y*x$.

The left-hand side of \eqref{eq09} vanishes for any associative algebra. Therefore, the right-hand side of \eqref{eq09} is a linear in $\{x_s\}$ element in $A$ which is 0. This element belongs to $\Gamma_2$ and defines 0 in $\Gamma_2/\Gamma_3$. Therefore,
\begin{equation}
\Gamma_2/\Gamma_3=\hbar^2\cdot (S(V)/(\sum_{a,b}(c_{ij}^ac_{ak}^b+c_{jk}ac_{ai}^b+c_{ki}^ac_{aj}^b)\cdot x_b)_{i,j,k})
\end{equation}
is the quotient by the ideal generated by the r.h.s. of \eqref{eq09} for all $i,j,k$.

\subsection{}
This paper is a much elaborated version of our earlier preprint [Sh2].

Along with more clear exposition of ideas of [Sh2], this paper contains several new things, listed below.

(1) We upgrade Lemma \ref{lemma21} (which is quoted from [Sh2]) to a more powerful Theorem \ref{lemma21bis}; roughly speaking, now it is not necessary to perturb all the components of the differential, but only the two most extreme ones, $d(\xi_{ij})$ and $d(\xi_{ijk})$ (among which the component $d(\xi_{ij})$ is given automatically).

(2) We give an application of our method to a new proof of results of Etingof-Ginzburg [EG], on the PBW property of non-commutative algebras ``with cyclic cubic potential'', for generic parameters. Our proof is contained in Section \ref{section32}.

(3) As another application, we study the case of a ``general'' quadratic algebra, as follows.
Let
\begin{equation}
\varphi_{ij}=\hbar\sum_{a,b}\alpha_{ij}^{ab}x_a\otimes x_b
\end{equation}
with $\alpha_{ij}^{ab}\in k$, where $\alpha_{ij}^{ab}=-\alpha_{ji}^{ab}$ {\it but in general $\alpha_{ij}^{ab}\ne\alpha_{ij}^{ba}$}.
Consider the algebra
$A_\alpha$ defined as the quotient algebra
$$
A_\alpha=T(V)[\hbar]/(x_i\otimes x_j-x_j\otimes x_i-\varphi_{ij})_{i,j}
$$
We prove that this algebra is PBW for generic specialization $\hbar=a\in k$, if the following identity holds: for any $i,j,k,b,c,d$
\begin{equation}\label{ux}
\Cycl_{i,j,k}\sum_s (\alpha_{jk}^{sb}\alpha_{is}^{cd}+ \alpha_{jk}^{cs}\alpha_{is}^{db})=0
\end{equation}
Note that \eqref{ux} implies the Poisson condition
$$
\Cycl_{ijk}\Symm_{abc}\sum_s\beta_{is}^{ab}\beta_{jk}^{sc}=0
$$
for the bivector $\beta=\sum_{ij}\beta_{ij}\partial_i\wedge\partial_j$ where $$\beta_{ij}=\sum_{a,b}(\alpha_{ij}^{ab}+\alpha_{ij}^{ba})x_ax_b$$
This result seems to be new.

\subsubsection*{\sc Acknowledgements}
I am grateful to Pavel Etingof, Boris Feigin, Victor Ginzburg, Volodya Rubtsov, and Vadik Vologodsky, for many fruitful discussions. I would like to mention that it is Volodya Rubtsov whose many questions on my previous preprint [Sh2] inspired and stimulated me to develop these ideas further, in attempt of answering them, which have been resulted in the current paper. As well, I would like to mention that Pavel Etingof was the first to whom I communicated my proof of PBW for the Etingof-Gingburg algebras, and his inspiration was another good stimulus to continue this work. As well, many of Pavel's explanations and remarks were very helpful. This work, whose first ``take'' was published as loc.cit., had being started in Luxembourg around 2007. I am thankful to the University of Luxembourg and personally to Martin Schlichenmaier, for excellent working conditions and for financial support, throughout the research
grant R1F105L15 of the University of Luxembourg.

\section{\sc General theory}

\subsection{\sc The PBW property}
Let $k$ be a field, $V=\{x_1,\dots,x_n\}$ a finite-dimensional vector space over $k$ with basis $\{x_1,\dots, x_n\}$,
$n=\dim_k V$.
Denote by $T(V)$ the free associative algebra over $k$ with generators $V$, it is identified with the tensor algebra $T(x_1,\dots,x_n)$ in variables $x_1,\dots,x_n$.

Let $\hbar$ be a formal parameter. For any pair $(i,j)$, $1\le i<j\le n$, choose an element $\varphi_{ij}\in\hbar\cdot T(V)[\hbar]$.
Suppose that $\deg_{x}\varphi_{ij}\le 2$.
The main object of our study is the associative algebra
\begin{equation}\label{eq1}
A=T(V)[\hbar]/(x_i\otimes x_j-x_j\otimes x_i-\varphi_{ij})_{1\le i<j\le n}
\end{equation}
which is the $\hbar$-linear quotient-algebra of the free algebra $T(V)[\hbar]$ by the two-sided ideal, generated by
\begin{equation}\label{eq2}
x_i\otimes x_j-x_j\otimes x_i-\varphi_{ij}
\end{equation}
for all $1\le i<j\le n$. We set also $\varphi_{ij}=-\varphi_{ji}$ if $i>j$, and $\varphi_{ii}=0$.

With our condition $\deg_{x}\varphi_{ij}\le 2$, the algebra $A$ is filtered with an ascending filtration $\{F_k\}$.
By definition, $F_k$ is the image under the natural projection $p\colon T(V)[\hbar]\to A$ of the vector space $W_k\subset T(V)[\hbar]$ spanned by monomials whose total degree by all $\{x_i\}$ is $\le k$. It is an algebra filtration, that is
\begin{equation}\label{eq3}
F_i*F_j\subset F_{i+j}
\end{equation}
where $*$ is the product in $A$.

It allows to consider the associated graded algebra
\begin{equation}\label{eq4}
\gr_F A=\oplus_{i\ge -1}F_{i+1}/F_i
\end{equation}
where by definition $F_{-1}=0$.

\begin{defn}\label{pbwd1}
{\rm The associative algebra $A$ defined in \eqref{eq1} is said to satisfy {\it the Poincar\'{e}-Birkhoff-Witt (PBW) property} if there is graded $\hbar$-linear isomorphism
\begin{equation}\label{eq7}
\bigoplus_{\ell\ge 0}F_\ell A/F_{\ell-1}A\simeq S(V)[\hbar]
\end{equation}
such that each consecutive quotient $F_\ell/F_{\ell-1}$ is $\hbar$-linearly isomorphic
to the $\ell$-th symmetric power $S^\ell(V)[\hbar]$.}
\end{defn}

A closely related property is the PBW property for algebras over $k$. In our study, these algebras will appear as the specializations of $k[\hbar]$-linear algebras.

Let $V=\{x_1,x_2,\dots\}$, and let $\psi_{s}\in T(V)$ be non-commutative polynomials of degree $\le 2$. Consider the associative algebra
\begin{equation}\label{eq5}
B=T(V)/(\psi_s)
\end{equation}
The algebra $B$ is endowed with natural ascending filtration (which is denoted also by $F$), where $F_k$ is formed by the elements of degree $\le k$ in $\{x_i\}$. Denote $F_{-1}=0$.

\begin{defn}\label{pbwd1bis}{\rm
The associative algebra $B$ defined in \eqref{eq5} is said to satisfy the PBW property if there is an isomorphism of graded algebras:
\begin{equation}\label{eq8}
\bigoplus_{\ell\ge 0}F_\ell/F_{\ell-1}\simeq S(V)
\end{equation}
}
\end{defn}
The following Lemma is very elementary. We include it here for further references.
\begin{lemma}\label{trivial}
Let $A$ be a $k[\hbar]$-algebra as in \eqref{eq1}. Suppose the algebra $A$ satisfies the $k[\hbar]$-linear PBW condition of Definition \ref{pbwd1}. Then for an $a\in k$, the specialization
\begin{equation}
A_a:=A\otimes_{k[\hbar]}k[\hbar]/(\hbar-a)
\end{equation}
is an algebra $B$ as in \eqref{eq5}, with
\begin{equation}\label{eq9}
\psi_{ij}=x_i\otimes x_j-x_j\otimes x_i-\varphi_{ij}(a)
\end{equation}
 Moreover, any specialization $A_a$, $a\in k$, satisfies the PBW condition of Definition \ref{pbwd1bis}.
\end{lemma}

\begin{proof}
We need to know that $A$ is a free $k[\hbar]$-module, which follows from Definition \ref{pbwd1}, saying that $\gr_F A$ is.
Now $A$ is a quotient $k[\hbar]$-module $A=T(V)[\hbar]/I[\hbar]$ where $I$ is the two-sided ideal. As $A$ is a free (therefore, a flat) $k[\hbar]$-module, the specialization $A_a$ is the quotient $T(V)[\hbar]_a/I[\hbar]_a$ of specializations, whence the first assertion follows.
To prove that $A_a$ is a PBW as in Definition \ref{pbwd1bis}, we first mention that $F_\ell(A)/F_{\ell-1}(A)$ is a free $k[\hbar]$-module, by \eqref{eq7}. Therefore, $(F_\ell(A)/F_{\ell-1}(A))_a\simeq F_\ell(A)_a/F_{\ell-1}(A)_a$. The right-hand side is $F_\ell(B)/F_{\ell-1}(B)$, with $B=A_a$, by the first assertion of Lemma.
\end{proof}

\subsection{\sc First Main Theorem}
The main topic of this paper is a rather general method, which reduces the proving of PBW property to another, more easily checked, property.
Let us formulate it.

Along with the {\it ascending} filtration $\{F_k\}$, the associative algebra $A$ admits the natural {\it descending} filtration $\{\Gamma\}$ by powers of $\hbar$:
\begin{equation}
\Gamma_\ell=\hbar^\ell A
\end{equation}
It is also an algebra filtration:
\begin{equation}
\Gamma_i*\Gamma_j\subset \Gamma_{i+j}
\end{equation}
which allows to define the associated graded algebra
\begin{equation}
\gr_\Gamma A=\oplus_{i\ge 0} \Gamma_i/\Gamma_{i+1}
\end{equation}

\begin{defn}\label{pbwd2}{\rm
The associative algebra $A$ is said to satisfy {\it the PBW-like property for the descending filtration $\Gamma$} if $\oplus_{\ell\ge 0}\Gamma_\ell A/\Gamma_{\ell+1}A\simeq S(V)[\hbar]$ as a graded $k[\hbar]$-algebra, such that $\Gamma_\ell A/\Gamma_{\ell+1}A\simeq \hbar^\ell S(V)$.
}
\end{defn}

Our main result in Section 1 is:
\begin{theorem}\label{main1}
Let $k$ be a field and let $A$ be as above. Suppose $A$ satisfies the PBW-like property for the descending filtration $\Gamma$, see Definition \ref{pbwd2}. Then:
\begin{itemize}
\item[(1)]
Suppose that
\begin{equation}\label{s1}
\bigcap_{\ell\ge 0}\Gamma_\ell=0
\end{equation}
Then the algebra $A$ obeys the PBW property of Definition \ref{pbwd2}, with respect to the ascending filtration $\{F_k\}$. Its specialization $$A_a=A\otimes_{k[\hbar]}k[\hbar]/(\hbar-a)$$ is a PBW algebra in the sense of Definition \ref{pbwd1bis} and is $T(V)/(\psi_{ij})$, with $\psi_{ij}$ as in \eqref{eq9}, for any $a\in k$.
\item[(2)]
Suppose that the vector space $V=\{x_1,x_2,\dots\}$ in \eqref{eq1} is finite-dimensional over $k$. Then in general, when $\bigcap_{k\ge 0}\Gamma_k\ne 0$, its support $\mathscr{S}$ (as of a $k[\hbar]$-module) is at most a countable set of points $s\in k$. For any $a\in k\backslash \mathscr{S}$, the specialization $A_a=A\otimes_{k[\hbar]}k[\hbar]/(\hbar-a)$ is a PBW algebra.
\end{itemize}
\end{theorem}

\begin{remark}{\rm
Equation \eqref{s1} always holds when all $\varphi_{ij}$ are linear in $x_i$'s, see Lemma \ref{pbwlemma} below. If $\{\varphi_{ij}\}$ are quadratic in $x_i$'s, \eqref{s1} is not true, in general. See Section \ref{exx} below, where such phenomenon is shown for Etingof-Ginzburg algebras with cyclic cubic potential.
}
\end{remark}

\subsection{\sc Proof of Theorem \ref{main1}}

\subsubsection{\sc Lemma on two filtrations}
\begin{lemma}\label{l2f}
Let $R$ be a ring, and $L$ an $R$-module. Suppose that $L$ is endowed with an ascending filtration $F$, and with a descending filtration $\Gamma$.
Then the filtration $F$ induces an ascending filtration on each consecutive quotient $\gr_\Gamma^i L$, and the filtration $\Gamma$ induces a descending filtration on each consecutive quotient $\gr_\Gamma^j L$. Their consecutive quotients are canonically isomorphic:
$$
\gr_\Gamma^i\gr_F^jL\simeq \gr_F^j\gr_\Gamma^i L
$$
\end{lemma}
\begin{proof}
It is clear, as both $R$-modules are isomorphic to
$$
\Gamma_i\cap F_j/(\Gamma_{i+1}\cap F_j+\Gamma_i\cap F_{j-1})
$$
\end{proof}
It is clear that the same statement is true when both filtrations are ascending or descending.
\subsubsection{\sc Proof of Theorem \ref{main1}(i)}
We need to prove that, under the assumption of Theorem \ref{main1}, the quotient $F_kA/F_{k-1}A$ is $\hbar$-linearly isomorphic to $S^k(V)[\hbar]$.
This quotient $F_kA/F_{k-1}A$ is filtered by descending filtration $\Gamma$, and by Lemma \ref{l2f}
\begin{equation}\label{x1}
\gr^j_\Gamma (F_kA/F_{k-1}A)\simeq\gr^k_F \Gamma_jA/\Gamma_{j+1}A\simeq\gr^k_F(\hbar^k S(V)[\hbar])\simeq\hbar^j S^k(V)
\end{equation}
as $k[\hbar]$-modules.
(Here in the second equality we used the assumption of Theorem \ref{main1}, on the PBW-like property for the descending filtration $\Gamma$).
As well, we know from these assumptions that $\oplus_{j\ge 0}\gr^j_\Gamma \gr^k_FA\simeq S^k(V)[\hbar]$ is a free $k[\hbar]$-module (for a fixed $k$).
Therefore, the $k[\hbar]$-module $\gr^k_FA/((\cap_j\Gamma_j)\cap \gr^k_FA)$ is also free. It follows that the space of its generators (as a $k[\hbar]$-module) is $\gr^0_\Gamma\gr^k_FA\simeq S^k(V)$, and then $\gr^k_FA/((\cap_j\Gamma_j)\cap\gr^k_FA)$ is isomorphic to the free $k[\hbar]$-module $S^k(V)[\hbar]$.

Now, as $\bigcap_{j\ge 0}\Gamma_j=0$ by assumption of (i), it follows that $F_kA/F_{k-1}A\simeq S^k(V)[\hbar]$, as an $k[\hbar]$-module.

The remaining assertions of Theorem \ref{main1}(i) follow now from Lemma \ref{trivial}.

\qed

\subsubsection{\sc The structure of a finitely-generated $k[\hbar]$-module}
Before proving the part (ii) of Theorem, recall the structure theorem for finitely-generated $k[\hbar]$ (resp., $k[[\hbar]]$) modules.
This result is well-known, and we refer the reader to the textbooks in Algebra for a proof.

\begin{lemma}\label{simple}
\begin{itemize}
\item[(i)] Let $k$ be any field. Any finitely generated $k[[\hbar]]$-module is a direct sum
\begin{equation}
M=M_\free\oplus M_{n_1}\oplus\dots\oplus M_{n_\ell}
\end{equation}
where $M_\free$, is a free $k[[\hbar]]$-module, $n_i\in\mathbb{Z}_{>0}$, and
\begin{equation}
M_n=k[[\hbar]]/\hbar^nk[[\hbar]]
\end{equation}
\item[(ii)] Let $k$ be an algebraically closed field. Any finitely generated $k[\hbar]$-module $M$ is a direct sum
\begin{equation}
M=M_\free \oplus M_{a_1, n_1}\oplus \dots \oplus M_{a_\ell, n_\ell}
\end{equation}
where $M_\free$ is a free $k[\hbar]$-module of a finite rank, $a_i\in k$, $n_i\in \mathbb{Z}_{>0}$, and
\begin{equation}
M_{a,n}=k[\hbar]/(\hbar-a)^nk[\hbar]
\end{equation}
\item[(iii)] When $k$ is not necessarily an algebraically closed field, a finitely generated $k[\hbar]$-module $M$ is a direct sum
\begin{equation}\label{kpl}
M=M_\free\oplus M_{p_1(\hbar),n_1}\oplus\dots\oplus M_{p_\ell(\hbar),n_\ell}
\end{equation}
where $M_\free$ is a free $k[\hbar]$-module, $p_i(\hbar)$ are (monic) irreducible polynomials in $k[\hbar]$, $n_i\in\mathbb{Z}_{>0}$, and
\begin{equation}
M_{p(\hbar),n}=k[\hbar]/p(\hbar)^nk[\hbar]
\end{equation}
\end{itemize}
\end{lemma}
\begin{proof}
It is standard. Notice that the statement (ii) is nothing but the Jordan normal form theorem, for the case of an algebraically closed field, and the statement (iii) is its direct analogue for general (not necessarily algebraically closed) fields. The modules $M_n$ are simple $k[[\hbar]]$-modules, as well as the modules $M_{a,n}$ (corresp., $M_{p(\hbar),n}$ for irreducible $p(\hbar)$) are free $k[\hbar]$-modules. Together with the 1-dimensional free module, they exhaust the list of finitely-generate simple modules over the corresponding commutative $k$-algebras ($k[[\hbar]]$ and $k[\hbar]$). \end{proof}

\comment

As well, one has:

\begin{lemma}\label{simple2}
Let $M$ be a finitely generated $k[\hbar]$-module, and in notations of Lemma \ref{simple}(iii), one has:
\begin{equation}
M=M_\free\oplus M_{p_1(\hbar), n_1}\oplus\dots\oplus M_{p_\ell(\hbar), n_\ell}
\end{equation}
with
\begin{equation}
M_\free =k[\hbar]^{\oplus N}
\end{equation}
Then the following statements are true:
\begin{itemize}
\item[(i)]
Let $a\in k$ be not a root of either of polynomials $p_i(\hbar)$. Then the specialization of the $k[\hbar]$-module $M$ at $a\in k$ is a $k$-vector space of dimension $N$.
\item[(ii)]
When $a\in k$ is a root of some of polynomials $p_i(\hbar)$, the specialization $M|_{\hbar=a}$ is
\begin{equation}
M|_{\hbar=a}=k^{\oplus N}\bigoplus_{\substack{1\le j\le \ell\\ p_j(\hbar)=\hbar-a}}k^{\oplus n_j}
\end{equation}
\item[(iii)]
The completion
$\hat{M}=M\otimes_{k[\hbar]}k[[\hbar]]$ is
\begin{equation}
\hat{M}=k[[\hbar]]^{\oplus N}\bigoplus_{\substack{1\le j\le \ell\\ p_j(\hbar)=\hbar}}M_{n_j}
\end{equation}
\item[(iv)]
The functor $M\mapsto \hat{M}$ is exact.
\end{itemize}
\end{lemma}
\begin{proof}
The first three statements are clear. The fourth one is clear in our situation from (iii), but it is a particular case of a more general statement, see [M, Theorem 8.1(ii)].
\end{proof}
\endcomment

\subsubsection{\sc Proof of Theorem \ref{main1}(ii)}
The same argument as in the proof of Theorem \ref{main1}(i) shows that, in general case,
\begin{equation}
\gr^p_F(A)/(\gr^p_F(A)\cap I)\simeq S^p(V)[\hbar]
\end{equation}
for any $k\ge 0$, where
\begin{equation}
I=\bigcap_{\ell\ge 0}\Gamma_\ell(A)
\end{equation}

We consider $I\cap \gr^p_F(A)$ as a $k[\hbar]$-module. It is finitely generated for any fixed $p$, as $F_p(A)$ is generated by (a finite number) of monomials in $x_1,\dots,x_n$ of total degree $\le p$. (Here we use essentially that $\dim V<\infty$, which is an assumption in (ii); otherwise $F_p$ would fail to be a finitely generated $k[\hbar]$-module).

Therefore, Lemma \ref{simple}(iii) is applicable to $M=I\cap \gr^p_F(A)$. We want to prove that the free component $M_\free$ in \eqref{kpl} is 0.
It is enough to prove that the specialization $M_0$ of $M$ at $\hbar=0$ is 0. We have: $M_0=M\otimes_{k[\hbar]}(k[\hbar]/(\hbar))$.
It is 0, as $I\otimes_{k[\hbar]}(k[\hbar]/(\hbar))=(\hbar I)\otimes_{k[\hbar]}(k[\hbar]/(\hbar))=0$ (we use $I=\hbar I$).

Therefore, for any fixed $p$,
\begin{equation}
M=I\cap \gr^p_F(A)=M_{p_1(\hbar),n_1}\oplus\dots\oplus M_{p_\ell(\hbar),n_\ell}
\end{equation}
for finite $\ell$.

Its specialization $M_s=0$ for any $s\in k$ which is not a root of either of polynomials $p_k(\hbar)$. Denote by $\mathscr{S}$ the set of $s\in k$ such that $p_i(s)=0$ for some $i$.

For any $t\in k\backslash \mathscr{S}$ one has:
\begin{equation}
\left(\gr^p_F(A)/(I\cap \gr^p_F(A))\right)_t=(\gr^p_F(A))_t/(I\cap \gr^p_F(A))_t=(\gr^p_F(A))_t
\end{equation}
as $\gr^p_F(A)/(I\cap \gr^p_F(A))\simeq S^p(V)[\hbar]$ is a free $k[\hbar]$-module, and the corresponding $\mathrm{Tor}_{k[\hbar]}^1=0$.

We obtain that for any $t\in k\backslash \mathscr{S}S$
\begin{equation}
\gr_F^p(A)_t\simeq S^p(V)[\hbar]/(\hbar-t)\simeq S^p(V)
\end{equation}

What only remains is to prove that $(F_p(A)/F_{p-1}(A))_t\simeq (F_p(A))_t/(F_{p-1}(A))_t$ for any $t\in k\backslash \mathscr{S}$.
We apply once again Lemma \ref{simple}(iii), which says that
$$
F_p(A)/F_{p-1}(A)=S^p(V)[\hbar]\bigoplus (\oplus M_{\theta_i(\hbar),m_i})
$$
where the second summand has support in $\mathscr{S}$. The corresponding $\mathrm{Tor}_{k[\hbar]}^1=0$, and we are done.

\qed

\section{\sc A way to prove the descending filtration PBW property}
Here we formulate and prove an effective tool allowing (in some cases) to prove that the assumption of Definition \ref{pbwd2} (and Theorem \ref{main1}) is fulfilled. That is, we establish a way to prove that $\oplus_{j\ge 0}\Gamma_jA/\Gamma_{j+1}A\simeq S(V)[\hbar]$ as a graded $k[\hbar]$-algebra.

\subsection{\sc The Koszul free resolution of the algebras of polynomials}
Here we recall the construction of a free dg resolution of the polynomial algebra $S(V)$, called {\it the Koszul resolution}. The vector space $V$ is either finite-dimensional over $k$, or graded with finite-dimensional graded components.

Denote by $\Lambda^-(V)$ the cofree commutative cocommutative coalgebra without counit on the space $V[1]$ (concentrated in degree -1 if $V$ is in degree 0).
That is, as a vector space, $\Lambda^-(V)=S^+(V[1])$, and the coproduct is
\begin{equation}\label{res1}
\begin{aligned}
\ &\Delta(v_1\wedge\dots\wedge v_\ell)=\sum_{\substack{\ell=a+b\\ a,b>0}}\sum_{\substack{(a,b)\text{-shuffles }\\ \sigma\in\Sigma_\ell}}(-1)^{\sharp\sigma}(v_{\sigma_1}\wedge\dots\wedge v_{\sigma_a})\bigotimes (v_{\sigma_{a+1}}\wedge\dots\wedge v_{\sigma_\ell})
\end{aligned}
\end{equation}
Now shift $\Lambda^-(V)$ by 1 to the right, and take the tensor algebra of this graded vector space:
\begin{equation}\label{res2}
\mathcal{R}^\udot=T(\Lambda^-[-1])
\end{equation}
Introduce a differential $d$ in $\mathcal{R}^\udot$ (as in a dg algebra, satisfying the Leibniz rule with signs), defining it on generators $(v_{i_1}\wedge\dots\wedge v_{i_k})[-1]$ by
\begin{equation}\label{res1bis}
d((v_{i_1}\wedge\dots\wedge v_{i_k})[-1])=\sum_i(\Delta^1_i(v_{i_1}\wedge\dots\wedge v_{i_k}))[-1]\otimes (\Delta^2_i(v_{i_1}\wedge\dots\wedge v_{i_k}))[-1]
\end{equation}
where
\begin{equation}
\Delta(\omega)=\sum_i\Delta^1_i(\omega)\otimes \Delta^2_i(\omega)
\end{equation}
is the notation for the coproduct.

The equation $d^2=0$ follows from the coassociativity of $\Delta$.

Choose a basis $\{x_1,\dots,x_n,\dots\}$ in $V$. Denote by $\xi_1,\dots,\xi_n,\dots$ the corresponding elements (the cogenerators) in $\Lambda^-(V)$. Then $\deg\xi_i=-1$. After the shift by 1, $\Lambda^-(V)[-1]$, we use the notations $x_i$ for $\xi_i[-1]\in\mathcal{R}^0$.
Denote $\xi_{ij}=\xi_i\wedge \xi_j\in (\mathcal{R})^{-1}$. Then \eqref{res1bis} gives
\begin{equation}
dx_i=0
\end{equation}
\begin{equation}
d(\xi_{ij})=x_i\otimes x_j-x_j\otimes x_i
\end{equation}
A general element in $(\mathcal{R})^{-1}$ is a sum of the following elements $\omega\in T(x_1,\dots,x_n,\dots)\otimes \xi_{ij}\otimes T(x_1,\dots,x_n,\dots)$.

By Leibniz rule: $d(\omega)\in T(V)\otimes (x_i\otimes x_j-x_j\otimes x_i)\otimes T(V)$.

Thus the image $d(\mathcal{R}^{\udot})\subset \mathcal{R}^0$ is the two-sided ideal generated by $x_i\otimes x_j-x_j\otimes x_i$.
Therefore,
\begin{equation}
H^0\mathcal{R}^\udot=S(V)
\end{equation}
as an algebra.

In fact, the higher cohomology of $\mathcal{R}^\udot$ vanish.
\begin{lemma}
The higher cohomology $H^k\mathcal{R}^\udot=0$ for any $k\le -1$.
\end{lemma}
It is standard in theory of Koszul algebras, see e.g. [PP] or [BGS]. In fact, one can construct a free ``Koszul'' resolution of any Koszul algebra $A$, which coincides with the one we just described for the case $A=S(V)$.

We will often use the formula for $d(\xi_{ijk})$ where $\xi_{ijk}=\xi_i\wedge \xi_j\wedge \xi_k \in\mathcal{R}^{-2}$.
It can be easily seen from \eqref{res2} that
\begin{equation}
d(\xi_{ijk})=[x_i,\xi_{jk}]+[x_j,\xi_{ki}]+[x_k,\xi_{ij}]
\end{equation}
Then applying $d$ once again, we see that
\begin{equation}
d^2(\xi_{ijk})=[x_i,[x_j,x_k]]+[x_j,[x_k,x_i]]+[x_k,[x_i,x_j]]
\end{equation}
which is 0 by the Jacobi identity.

\comment
In general, an associative algebra $A$ may have only ``very big'' quasi-free (cofibrant) resolutions $R$, like $R=\Cobar(\Bar(A_+)^-)$.
Here $A_+$ is the kernel of the augmentation map, and an augmentation map $\varepsilon\colon A\to k$ is a necessary data for construction of such resolution. It should obey the following property. Denote $A_+=\Ker \varepsilon$, and $F_n=A_+*\dots *A_+\subset A_+$ ($n$ of factors $A_+$, $*$ is the product in $A$). The property reads: $\cap_{n\ge 0}F_n=0$.

In a special case when the associative algebra $A$ is Koszul, it is possible to find a resolution $R_K(A)$ which is ``much smaller'', we call it {\it the Koszul free resolution} of $A$.
Here we recall the construction of this resolution and consider in more detail the case $A=S(V)$ which is of paramount importance for the sequel.

We refer the reader to [PP, Chapter 2] or [BGS, Section 2] for
definition and for the main properties of quadratic and Koszul algebras.

Let $A$ be a quadratic Koszul algebra, and let
\begin{equation}\label{cohk}
A^!=\oplus_{i\ge
0}\Ext^i_{\Mod(A)}(k,k)
\end{equation}
be the Koszul dual algebra. We refer to the algebra $A^!$ as the {\it cohomological} Koszul dual algebra.
We also consider the algebra
\begin{equation}\label{clask}
{A}^{!}_\mathrm{class}=\oplus_{i\ge
0}\Ext^i_{\Mod(A)}(k,k)[-i]
\end{equation}
(with the shifts of cohomological
degree),
and refer to it as the {\it classical} Koszul dual to $A$ algebra. When $A$ is concentrated in cohomological degree 0, the classical Koszul dual algebra $A^!_{\mathrm{class}}$ also is.

Denote
${A}^{!+}=\oplus_{i>0}\Ext^i_{\Mod(A)}(k,k)$. Then ${A}^{!+}$ is a non-unital $k$-algebra.

Suppose all graded components of ${A}^{!}$ are
finite-dimensional (this is always the case when the quadratic Koszul generators $V$ for $A$ form a finite-dimensional $k$-vector space). Then the associative non-counital {\it coalgebra}
$({A}^{!+})^*$ is well-defined, with the dual space taken grading-wise. Denote
\begin{equation}\label{v3.1}
R^\udot(A)=\Cobar(({A}^{!+})^*)
\end{equation}

The following standard Lemma is very important for our paper. We recall its proof for completeness.
\begin{lemma}\label{lemmares}
In the above notations, $R^\udot(A)$ is a free $\mathbb{Z}_{\le 0}$-graded dg algebra over $k$, endowed with the natural quasi-isomorphic surjection
$p\colon R^\udot(A)\to A$. That is, $R^\udot(A)$ is a free dg resolution of algebra $A$.
\end{lemma}
\begin{proof}
One checks immediately that 0-th cohomology of $R^\udot(A)$ is canonically isomorphic to $A$ as an algebra.
The vanishing of the higher cohomology is proven as follows.

There is the following general construction of resolutions of modules.
Let $B$ be an associative $k$-algebra with unit, $M$ a left $B$-module. Then the canonical map $B\otimes _k M\to M$, $b\otimes m\mapsto b\cdot m$
is surjective. It can be supplemented to a resolution of $M$ by free $B$-modules. This resolution is:
\begin{equation}\label{classcobar}
\dots\rightarrow (B/k)^{\otimes 2}\otimes B\otimes M\rightarrow B/k\otimes B\otimes M\rightarrow B\otimes M\rightarrow M\rightarrow 0
\end{equation}
The differential in \eqref{classcobar} is the bar-differential, and we omit the explicit formula for it. See ??? for the formula, and for a proof of its acyclicity in higher degrees.

Now apply the resolution \eqref{classcobar} to the case $B={A}^{!}$, $M=k$ the 1-dimensional module, defined out of the augmentation map.
We can compute cohomology $\Ext^\udot_{\Mod({A}^!)}(k,k)$ using this resolution, and the result a priori is the Koszul dual to ${A}^{!}$ algebra $A$ in degree 0, and the cohomology vanish in higher degrees. On the other hand, the explicit computation with resolution $R^\udot(A)$ (considered as a complex of $k$-vector spaces). It proves that the higher cohomology of $R^\udot(A)$ vanish.
\end{proof}

The constructed free resolution of a Koszul algebra $A$ is called {\it free Koszul resolution}, or {\it cobar resolution} of $A$, and is denoted by $R^\udot(A)$, or by $\Cobar^\udot(A)$.

In this paper, we use this construction only for $A=S(V)$. We need to know how the above resolution looks explicitly, so we consider here the simplest non-trivial cases $n=2$ and $n=3$, $n=\dim V$.

(When $V$ is 1-dimensional, the algebra $S(V)=k[x]$ is free as associative algebra, and $R^\udot(k[x])$ coincides with $k[x]$ itself).

\begin{example}\label{n=2}{\rm
Let $\dim_kV=2$, and we write $S(V)$ as $k[x_1,x_2]$. As a graded algebra, ${R}^\udot(k[x_1,x_2])$
is the free associative algebra $T(W)$ generated by a graded vector space $W$ with generators $x_1, x_2$ in degree 0 and with generator $\xi_{12}$ in degree -1. We write
${R}^\udot(k[x_1,x_2])=T(x_1,x_2,\xi_{12})$. The differential on generators is defined as $d(x_1)=d(x_2)=0$,
$d(\xi_{12})=x_1\otimes x_2-x_2\otimes x_1$. It is extended to the entire associative algebra $T(x_1,x_2,\xi_{12})$
by the graded Leibniz rule. The degree 0 component of $T(x_1,x_2,\xi_{12})$ is the tensor
algebra $T(x_1,x_2)$, and the differential restricted on it is 0. A general element of degree -1 is a
non-commutative string in $x_1,x_2,\xi_{12}$, with only one occurrence of $\xi_{12}$. The image of
the differential from degree -1 to degree 0 is the two-sided ideal in the tensor
algebra $T(x_1,x_2)$ generated by $x_1\otimes x_2-x_2\otimes x_1$.
Thus, 0-th cohomology is equal to $k[x_1,x_2]$. It follows from
Lemma \ref{lemmares} that the higher cohomology vanishes.}
\end{example}

\begin{example}\label{n=3}{\rm
Consider the case of the algebra $k[x_1,x_2,x_3]$. Then the underlying graded algebra
${R}^\udot(k[x_1,x_2,x_3])=T(x_1,x_2,x_3,\xi_{12},
\xi_{23},\xi_{13},\xi_{123})$ with $\deg x_i=0$, $\deg \xi_{ij}=-1$,
$\deg\xi_{123}=-2$. The differential is defined on generators as 0 on $x_1,x_2,x_3$,
$d(\xi_{ij})=x_i\otimes x_j-x_j\otimes x_i$, and
$d(\xi_{123})=(x_1\otimes \xi_{23}+x_2\otimes \xi_{31}+x_3\otimes
\xi_{12})+(\xi_{23}\otimes x_1+\xi_{31}\otimes x_2+\xi_{12}\otimes
x_3)$. Here we set $\xi_{ij}=-\xi_{ji}$. Then cohomology in degree 0
is equal to the quotient of the free algebra $T(x_1,x_2,x_3)$ by the
two-sided ideal generated by $x_i\otimes x_j-x_j\otimes x_i$, which
is the algebra $k[x_1,x_2,x_3]$. Again, follows from Lemma \ref{lemmares} that the higher cohomology
vanishes.}
\end{example}
\endcomment

\subsection{\sc A Lemma}
We start with some Lemma which is much weaker than Theorem \ref{lemma21bis} below. Nevertheless, it better explains ``what goes on'', so we decided to keep it here. This Lemma is quoted from [Sh2], whence Theorem \ref{lemma21bis} is new.
\begin{lemma}\label{lemma21}
Let $\mathcal{R}^\udot$ be a $\mathbb{Z}_{\le 0}$-graded complex with
differential $d_0$, such that $H^i(\mathcal{R}^\udot)$ vanishes for
all $i\ne 0$.
\begin{itemize}
\item[(i)] Consider
$\mathcal{R}^\udot_{\hbar}=\mathcal{R}^\udot\otimes k[\hbar]$.
Let $d_\hbar\colon \mathcal{R}^\udot_\hbar\to
\mathcal{R}^{\udot+1}_\hbar$ be an $\hbar$-linear map of degree
+1 $d_\hbar=d_0+\hbar d_1+\hbar^2 d_2+\dots \hbar^n d_n$ (a finite
sum) such that
$$
d_\hbar^2=0
$$
Denote by $H^\udot_\hbar$ the cohomology of the complex
$(\mathcal{R}^\udot_\hbar,d_\hbar)$. Consider the filtration
$\mathcal{R}^\udot_\hbar\supset\hbar\mathcal{R}^\udot_\hbar\supset\hbar^2\mathcal{R}^\udot_\hbar\supset\dots$,
and the induced filtration on $H^\udot_\hbar$: $F_i
H^\udot_\hbar=\mathrm{Im}(i\colon H^\udot(\hbar^i \mathcal{R}^\udot_\hbar,
d_\hbar)\to H^\udot(\mathcal{R}^\udot_\hbar, d_\hbar))$. Then there are
canonical isomorphisms of $k[\hbar]$-modules
$$
F_i H^0_\hbar/F_{i+1}H^0_\hbar\eqto \hbar^i H^0(\mathcal{R}^\udot,
d_0)
$$
and $F_i H^k_\hbar/F_{i+1}H^k_\hbar=0$ for $k<0$,
\item[(ii)] the same statement as in (i), for
$\mathcal{R}^\udot_{[[\hbar]]}=\mathcal{R}^\udot\otimes k[[\hbar]]$,
and $d_\hbar=d_0+\hbar d_1+\hbar^2 d_2+\dots$, possibly an infinite
sum, $d_\hbar^2=0$.
\end{itemize}
\begin{proof}
Consider the short exact sequences of complexes $S_k$:
\begin{equation}\label{lemma21new1}
0\rightarrow \hbar^{k+1}\mathcal{R}^\udot_\hbar\rightarrow
\hbar^k\mathcal{R}^\udot_\hbar\rightarrow
\hbar^{k}\mathcal{R}^\udot_\hbar/\hbar^{k+1}\mathcal{R}^\udot_\hbar\rightarrow
0
\end{equation}
The complex
$\hbar^{k}\mathcal{R}^\udot_\hbar/\hbar^{k+1}\mathcal{R}^\udot_\hbar$
has the differential $d_0$ because all its higher components vanish.
Consider the long exact sequence of cohomology corresponded to this short exact exact sequence of complex.
It has many zero terms, namely
$H^\ell(\hbar^{k}\mathcal{R}^\udot_\hbar/\hbar^{k+1}\mathcal{R}^\udot_\hbar)=0$
for $\ell\le -1$. Then the long exact sequence proves that the
imbedding
$\hbar^{k+1}\mathcal{R}^\udot_\hbar\hookrightarrow\hbar^k\mathcal{R}^\udot_\hbar$
induces an isomorphism on $\ell$-th cohomology for all $\ell\le -1$.
Consider the end fragment of the long exact sequence:
\begin{equation}\label{lemma21new2}
\begin{aligned}
\ &\dots\rightarrow H^1(\hbar^{k+1}\mathcal{R}^\udot_\hbar)\rightarrow
H^1(\hbar^k\mathcal{R}^\udot_\hbar)\rightarrow 0\rightarrow\\
&\rightarrow H^0(\hbar^{k+1}\mathcal{R}^\udot_\hbar)\rightarrow
H^0(\hbar^{k}\mathcal{R}^\udot_\hbar)\rightarrow
H^0(\hbar^{k}\mathcal{R}^\udot_\hbar/\hbar^{k+1}\mathcal{R}^\udot_\hbar)\rightarrow
0
\end{aligned}
\end{equation}
It proves all assertions of Lemma.
\end{proof}
\end{lemma}
\comment
\begin{remark}{\rm
Unlike in the case of Lemma, when the
complex $(\mathcal{R}^\udot, d_0)$ is $\mathbb{Z}_{\ge 0}$-graded (and
again only $H^0(\mathcal{R},d_0)\ne 0$),  the claim (i)
fails, but (ii) is still true.

By the long exact sequence arguments one needs to prove only the
surjectivity of the map $H^0(\hbar^k \mathcal{R}^\udot_{[[\hbar]]})\to
H^0(\hbar^k \mathcal{R}^\udot_{[[\hbar]]}/\hbar^{k+1}
\mathcal{R}^\udot_{[[\hbar]]})$. The proof (which is true only over
$k[[\hbar]]$) goes as follows.

We consider only the case $k=0$, the case of general $k$ is analogous. Let $x\in H^0(\mathcal{R}^\udot, d_0)$ be a cycle, where $\mathcal{R}^\udot=\mathcal{R}^\udot[[\hbar]]/\hbar\mathcal{R}^\udot[[\hbar]]$. One needs to
lift it to a cycle of the form $x+\hbar x_1+\hbar^2 x_2+\dots$ in
$H^0(\mathcal{R}^\udot_{[[\hbar]]}, d_\hbar)$. We a looking for
$$
x^{(k)}=x+\hbar x_1 +\hbar^2 x_2+\dots +\hbar^k x_k
$$
such that

$$
d_\hbar x^{(k)}=0\ \mathrm{mod} \ \hbar^{k+1}
$$
Note that $d_0((d_\hbar
x^{(k)})_{k+1})=0\ \mathrm{mod} \hbar^{k+1}$, as follows from the fact that
$d_\hbar^2(x^{(k)})=0$.

Now we perform the step of induction. Consider
$(d_\hbar(x^{(k)}))_{k+1}$. It is a $d_0$-cycle, and we can find $x_{k+1}\in
\mathcal{R}^0$ such that $d_0(x_{k+1})=(d_\hbar x^{(k)})_{k+1}$. Set
$$
x^{(k+1)}=x+\hbar x_1+\dots +\hbar^{k+1}x_{k+1}
$$
(because of vanishing of $H^1(\mathcal{R}^\udot)$).

Then we see that $d_\hbar(x^{(k+1)})=0\ \mathrm{mod}\ \hbar^{k+2}$.
}
\end{remark}
\endcomment

\subsection{\sc Second Main Theorem}
Here we generalize Lemma \ref{lemma21} to weaken the assumptions under which the algebra $H^0(\mathcal{R}^\udot_\hbar)$ enjoys the PBW property.
We remark in the proof of Lemma \ref{lemma21} we did not use that the differential $d_0$ can be ``perturbed'' in all degrees. In fact, all we need is to perturb the differential $d_0$ in degrees -1 and -2, such that $d_\hbar^{(-1)}\circ d_\hbar^{(-2)}=0$. The differential component $d_\hbar^{(-1)}$ is encoded in the non-commutative algebra we check for the PBW property, so the only what we need is to construct $d_\hbar^{(-2)}=d_0^{(-2)}+\mathcal{0}(\hbar)$, such that
\begin{equation}
d_\hbar^{(-1)}\circ d_\hbar^{(-2)}=0
\end{equation}

\begin{theorem}\label{lemma21bis}
Let $\mathcal{R}^\udot$ be a $\mathbb{Z}_{\le 0}$-graded complex with
differential $d_0$, such that $H^{-1}(\mathcal{R}^\udot)$ vanishes (without any assumptions on $H^{-2},H^{-3},\dots$).
\begin{itemize}
\item[(i)] Consider
$\mathcal{R}^\udot_{\hbar}=\mathcal{R}^\udot\otimes k[\hbar]$. For $\ell=-1$ and $\ell=-2$,
let $d_\hbar^{(\ell)}\colon \mathcal{R}^\udot_\hbar\to
\mathcal{R}^{\udot+1}_\hbar$ be an $\hbar$-linear map of degree
+1 $d_\hbar^{(\ell)}=d_0^{(\ell)}+\hbar d_1^{(\ell)}+\hbar^2 d_2^{(\ell)}+\dots \hbar^n d_n^{(\ell)}$ (a finite
sum) such that
$$
d_\hbar^{(-1)}\circ d_\hbar^{(-2)}\colon \mathcal{R}^{-2}_\hbar\to\mathcal{R}_\hbar^{0}=0
$$
Denote by $\overline{\mathcal{R}}_\hbar$ the following complex:
\begin{equation}
0\rightarrow \mathcal{R}_\hbar^{-2}\xrightarrow{d_\hbar^{(-2)}}\mathcal{R}_\hbar^{-1}\xrightarrow{d_\hbar^{(-1)}}\mathcal{R}^0_\hbar\rightarrow 0
\end{equation}
and denote by $H^\ell_\hbar$, $\ell=0,-1,-2$, its cohomology.

Consider the filtration
$\overline{\mathcal{R}}_\hbar\supset\hbar\overline{\mathcal{R}}_\hbar\supset\hbar^2\overline{\mathcal{R}}_\hbar\supset\dots$,
and the induced filtration on $H^\ell_\hbar$: $F_i
H^\ell_\hbar=\mathrm{Im}(i\colon H^\ell(\hbar^i \overline{\mathcal{R}}_\hbar,
d_\hbar)\to H^\ell(\overline{\mathcal{R}}^{\ell}_\hbar, d_\hbar))$. Then there are
canonical isomorphisms of $k[\hbar]$-modules
$$
F_i H^0_\hbar/F_{i+1}H^0_\hbar\eqto \hbar^i H^0(\mathcal{R}^\udot,
d_0)
$$
\item[(ii)] the same statement as in (i), for
$\mathcal{R}^\udot_{[[\hbar]]}=\mathcal{R}^\udot\otimes k[[\hbar]]$,
and $d_\hbar^{(\ell)}=d_0^{(\ell)}+\hbar d_1^{(\ell)}+\hbar^2 d_2^{(\ell)}+\dots$, possibly an infinite
sum, $d_\hbar^{(-1)}\circ d_\hbar^{(-2)}=0$.
\end{itemize}
\end{theorem}
\begin{proof}
The proof is literally the same as for Lemma \ref{lemma21}, as only what we used in it is the vanishing of $H^{-1}(\mathcal{R}^\udot,d_0)$.
Consider the short exact sequence of complexes:
\begin{equation}
0\rightarrow \hbar^{k+1}\overline{\mathcal{R}}_\hbar\rightarrow
\hbar^k\overline{\mathcal{R}}_\hbar\rightarrow
\hbar^{k}\overline{\mathcal{R}}_\hbar/\hbar^{k+1}\overline{\mathcal{R}}_\hbar\rightarrow
0
\end{equation}
Consider the most right fragment of the corresponding long exact sequence in cohomology:
\begin{equation}
\begin{aligned}
\ &\dots\rightarrow H^1(\hbar^{k+1}\overline{\mathcal{R}}_\hbar)\rightarrow
H^1(\hbar^k\overline{\mathcal{R}}_\hbar)\rightarrow 0\rightarrow\\
&\rightarrow H^0(\hbar^{k+1}\overline{\mathcal{R}}_\hbar)\rightarrow
H^0(\hbar^{k}\overline{\mathcal{R}}_\hbar)\rightarrow
H^0(\hbar^{k}\overline{\mathcal{R}}_\hbar/\hbar^{k+1}\overline{\mathcal{R}}_\hbar)\rightarrow
0
\end{aligned}
\end{equation}
In degrees $\ell=0,-1,-2$, $\overline{\mathcal{R}}_\hbar^\ell=\mathcal{R}_\hbar^\ell$. As well, in degree $\ell=-1$ one has
\begin{equation}
H^{-1}(\hbar^k\overline{\mathcal{R}}_\hbar/\hbar^{k+1}\overline{\mathcal{R}}_\hbar)=H^{-1}(\hbar^k\mathcal{R}^\udot_\hbar/\hbar^{k+1}\mathcal{R}^\udot_\hbar)=0
\end{equation}
for any $k\ge 0$, by the assumption. (It is not true for $\ell\le -2$). We are done.
\end{proof}

\subsection{\sc Conclusion}\label{concl}
Let
\begin{equation}\label{con0}
A=T(V)[\hbar]/(x_i\otimes x_j-x_j\otimes x_i-\varphi_{ij})_{1\le i<j\le n}
\end{equation}
be an associative algebra, with
$\varphi_{ij}\in\hbar T(V)[\hbar]$.

Consider the Koszul resolution \eqref{res2} of the polynomial algebra $S(V)$, denote it $\mathcal{R}^\udot$. We can perturb the $\hbar$-linear differential in $\mathcal{R}^\udot_\hbar=\mathcal{R}^\udot[\hbar]$ using the non-commutative polynomials $\varphi_{ij}$.

More precisely, we set
\begin{equation}\label{con1}
d^{(-1)}_\hbar(\xi_{ij})=d^{(-1)}_0(\xi_{ij})-\varphi_{ij}=x_i\otimes x_j-x_j\otimes x_i-\varphi_{ij}
\end{equation}
As $\varphi_{ij}=\mathcal{O}(\hbar)$, we are in the setting of Theorem \ref{lemma21bis}. Only what we need to prove that the condition of Definition \ref{pbwd2} is fulfilled for $A$, it to perturb the differential $d^{(-2)}$, that is, to define
\begin{equation}\label{con2}
d^{(-2)}_\hbar(\xi_{ijk})=d_0^{(-2)}(\xi_{ijk})+\mathcal{O}(\hbar)
\end{equation}
such that
\begin{equation}\label{con3}
d^{(-1)}_\hbar\circ d^{(-2)}_\hbar(\xi_{ijk})=0
\end{equation}
for any $i,j,k$.

We may not care on the higher components of the differential, in the sense that we do not need to perturb them.

If we succeed to construct, for given $\{\varphi_{ij}\}$, the linear map $d_\hbar^{(-2)}$ such that \eqref{con3} holds for any $i,j,k$, we know that Definition \ref{pbwd2} is fulfilled for the algebra $A$ in \eqref{con0}, and we can apply Theorem \ref{main1} to it.

We consider three examples how it works in Section 3.

\section{\sc Examples}\label{section3}

\subsection{\sc The classical Poincar\'{e}-Birkhoff-Witt theorem}\label{section31}
Let $\mathfrak{g}$ be a Lie algebra over a field $k$.
Consider the cocommutative coalgebra $\Lambda^-(\g):=S^+(\g[1])$. The Lie algebra structure on $\g$ defines the Chevalley-Eilenberg differential on $\Lambda^-(\g)$, which makes it a dg coalgebra over $k$.

There is the $\hbar$-linear version of this construction. Consider $\Lambda_\hbar^-(\g):=\Lambda^-(\g)[\hbar]$, as a dg coalgebra.
The coproduct $\Delta\colon \Lambda_\hbar^-(\g)\to S^2_{k[\hbar]}\Lambda^-_\hbar(\g)$ is defined from the coproduct on $\Lambda^-(\g)$ by $k[\hbar]$-linearity. The differential is defined as $d=\hbar d_\g$ where $d_\g$ is the Chevalley-Eilenberg differential in $\Lambda^-(\g)$.

We can now take the tensor algebra over $k[\hbar]$ $T_{k[\hbar]}(\Lambda^-\hbar(\g)[-1])$, it becomes a dg associative algebra.
Namely, the differential is the sum $d_\hbar=d_0+\hbar d_1$, where $d_0$ is defined in \eqref{res1bis} above, and $d_1$ is the Chevalley-Eilenberg differential (extended by the Leibniz rule).

The identity $(d_0+\hbar d_1)^2=0$ follows from the fact that $\Lambda_\hbar^-(\g)$ is an $\hbar$-linear dg coalgebra, as it just was defined.

This construction gives the perturbation of the differential in the Koszul resolution $T(S(\g))[\hbar]$, defined at once in all degrees.
In particular, $$d^{-1}(\xi_{ij})=x_i\otimes x_j-x_j\otimes x_i-\hbar [x_i,x_j]_\g$$
$$
d^{-2}(\xi_{ijk})=\Cycl_{ijk}(x_i\otimes \xi_{jk}-\xi_{jk}\otimes x_i)-\hbar \sum_p\Cycl_{ijk}c_{ij}^p\xi_{pk}
$$
We conclude that the condition of Definition \ref{pbwd2} is fulfilled for the algebra $A$ of type \eqref{eq1} with
$$
\varphi_{ij}=\hbar[x_i,x_j]_\g=\hbar\sum_kc_{ij}^kx_k
$$
This algebra $A$ is, by definition, the quotient-algebra
\begin{equation}
A=T(\g)[\hbar]/(x\otimes y-y\otimes x-\hbar[x,y]_\g)_{x,y\in\g}=\mathcal{U}(\g_\hbar)[\hbar]
\end{equation}

To apply Theorem \ref{main1}(i), we need the following elementary Lemma:

\begin{lemma}\label{pbwlemma}
Let $A$ be an algebra of type \eqref{eq1} such that all $\varphi_{ij}$ be linear in $\{x_i\}$. Then $\bigcap_{\ell\ge 0}\Gamma_\ell(A)=0$.
\end{lemma}
\begin{proof}
It is clear.
\end{proof}

Now Theorem \ref{main1}(i) gives the classical Poincar\'{e}-Birkhoff-Witt theorem (in its $\hbar$-linear version). We can specialize it at any $\hbar=a$ as $F_k(A)/F_{k-1}(A)\simeq S^k(V)[\hbar]$ is a free (and therefore a flat) $k[\hbar]$-module. It gives the genuine PBW theorem.

\subsection{\sc The Etingof-Ginzburg algebras with cyclic cubic potential}\label{section32}
Recall some basic definitions on cyclic words.

A {\it cyclic word} in variables $x_1,\dots,x_n$ is an element of the quotient-space $A_n/[A_n,A_n]$, where $A_n=T(x_1,\dots,x_n)$ is the free associative algebra with generators $x_1,\dots, x_n$. A {\it homogeneous of degree $d$} cyclic word in $x_1,\dots, x_n$ is an element of degree $d$ component in $A_n/[A_n,A_n]$, where the degree of all $x_i$ is equal to 1.

Let $\Phi$ be a cyclic word in $x_1,\dots, x_n$. Here we define the partial derivatives $\frac{\partial \Phi}{\partial x_i}$. These derivatives are not cyclic words, but elements of $A_n=T(x_1,\dots, x_n)$. By definition, for a single monomial cyclic word $\Phi$, $\frac{\partial \Phi}{\partial x_i}$ is a sum over all occurrences of $x_i$ in $\Phi$, for any such occurrence, we remove the corresponding $x_i$ from $\Phi$, and cut off the cyclic word $\Phi$ in the place of removed $x_i$. Then extend it to general cyclic words by linearity. One can easily see that the partial derivative $\frac{\partial}{\partial x_i}$ vanish on the commutant $[A_n,A_n]$, and defines an operation
\begin{equation}
\frac{\partial}{\partial x_i}\colon A_n/[A_n,A_n]\to A_n
\end{equation}

Let $\Phi$ be an arbitrary cyclic word in $A=A_3=T(x,y,z)$.
The Etingof-Ginzburg algebra with cyclic potential $\Phi$ is the associative algebra with generators $x,y,z$ and the defining relations
\begin{equation}
\begin{aligned}
\ &x\otimes y-y\otimes x=\frac{\partial \Phi}{\partial z}\\
&y\otimes z-z\otimes y=\frac{\partial\Phi}{\partial x}\\
&z\otimes x-x\otimes z=\frac{\partial\Phi}{\partial y}
\end{aligned}
\end{equation}
We use the [EG] notation $\mathscr{U}(\Phi)$ for these algebras.

The associative algebras $\mathscr{U}(\Phi)$ appeared in [EG] in constructing of ``non-commutative del Pezzo surfaces'' (see loc.cit., Sections 1.2-1.3).

Let $\hbar$ be a polynomial parameter. Consider an expression
\begin{equation}\label{phibar}
\Phi_\hbar=\hbar \Phi_1+{\hbar}^2 \Phi_2+\dots
\end{equation}
which is a polynomial in $\hbar$, all whose coefficients $\Phi_i$ are possibly non-homogeneous cyclic words (without any restriction on the degrees of homogeneity components) in $x,y,z$.

Define the $\hbar$-linear associative algebra $\mathscr{U}_\hbar(\Phi_\hbar)$ as the quotient-algebra of the $\hbar$-linear tensor algebra $T(x,y,z)[\hbar]$ by the two-sided ideal, generated by the following three elements:
\begin{equation}\label{tsi}
\begin{aligned}
\ & [x,y]-\frac{\partial \Phi_\hbar}{\partial z}\\
& [y,z]-\frac{\partial \Phi_\hbar}{\partial x}\\
& [z,x]-\frac{\partial \Phi_\hbar}{\partial y}
\end{aligned}
\end{equation}

\begin{lemma}\label{th3}
Let $\Phi_\hbar$ be as in \eqref{phibar}. Then the associative algebra $\mathscr{U}_\hbar(\Phi_\hbar)$  enjoys the PBW property for the descending filtration $\{\Gamma_\ell\}$ of Definition \eqref{pbwd2}.
\end{lemma}

\begin{proof}
Accordingly to Theorem \ref{lemma21bis} and discussion in Section \ref{concl} thereafter, it is enough to perturb the $(-2)$-component of the differential in the Koszul resolution $\mathcal{R}^\udot(x_1,x_2,x_2)$ such that the perturbation of the component $d^{(-1)}$ is
\begin{equation}\label{eg}
d_\hbar^{(-1)}(\xi_{ij})=x_i\otimes x_j-x_j\otimes x_i-\frac{\partial}{\partial x_{\overline{i,j}}}\Phi_\hbar
\end{equation}
where we use the notation $\overline{1,2}=3$, $\overline{2,3}=1$, $\overline{3,1}=2$, and assume that $x_{\overline{i,j}}=-x_{\overline{j,i}}$.

That is, the only what we need is to define a perturbation $d_\hbar^{(-2)}(\xi_{ijk})$ such that
\begin{equation}\label{w1}
d_\hbar^{(-1)}\circ d_\hbar^{(-2)}(\xi_{ijk})=0
\end{equation}

Our solution is do not perturb $d^{(-2)}$ at all. That is, we set
\begin{equation}\label{egbis}
d_\hbar(\xi_{123})=d_0(\xi_{123})=[\xi_{12},x_3]+[\xi_{23},x_1]-[\xi_{13},x_2]
\end{equation}
We need to check that this definition agrees with \eqref{w1}.
\begin{lemma}
Within the definitions \eqref{eg} and \eqref{egbis}, one has $d_\hbar^2(\xi_{123})=0$.
\end{lemma}
\begin{proof}
It is easy to see that
\begin{equation}\label{eg2}
d_\hbar^2(\xi_{123})=[\frac{\partial\Phi_\hbar}{\partial x_1},x_1]+[\frac{\partial\Phi_\hbar}{\partial x_2},x_2]+[\frac{\partial\Phi_\hbar}{\partial x_3},x_3]
\end{equation}
We claim that the sum in the right-hand side is 0 (for any cyclic word $\Phi_\hbar$).

Indeed, the first ``half'' of the sum in the r.h.s. of \eqref{eg2},
\begin{equation}\label{eg3}
\frac{\partial\Phi_\hbar}{\partial x_1}\otimes x_1+\frac{\partial\Phi_\hbar}{\partial x_2}\otimes x_2+\frac{\partial\Phi_\hbar}{\partial x_3}\otimes x_3
\end{equation}
is equal to the sum of all possible cuttings (at any place) of the cyclic word $\Phi_\hbar$ to an ordinary word.

The remaining second ``half'' of the right-hand side of \eqref{eg2},
\begin{equation}\label{eg4}
-x_1\otimes \frac{\partial\Phi_d}{\partial x_1}-x_2\otimes \frac{\partial\Phi_d}{\partial x_2}-x_3\otimes \frac{\partial\Phi_d}{\partial x_3}
\end{equation}
is equal to the same sum of all possible cuttings of the cyclic word $\Phi_d$ to an ordinary word, taken with the opposite sign.

The contributions of \eqref{eg3} and \eqref{eg4} cancel each other.
\end{proof}

Now Lemma \ref{th3} follows directly from Theorem \ref{lemma21bis}.
\end{proof}

Theorem \ref{main1}(ii) gives now the following result.
\begin{coroll}
Let $\Phi_\hbar$ be as in \eqref{phibar}, and let $A=\mathscr{U}_\hbar(\Phi_\hbar)$. Then for almost all $a\in k$ the specialization $A_a$ is a PBW algebra in the sense of Definition \ref{pbwd1bis}, and for all such $a\in k$ is the specialized algebra is isomorphic to $\mathscr{U}(\Phi(a))$, where $\Phi(a)$ is the value of $\Phi(\hbar)$ at $\hbar=a$.
\end{coroll}

We conclude the discussion of Etingof-Ginzburg algebras by an explicit computation, illustrating the assumption in Theorem \ref{main1}(ii): $\bigcap_\ell\Gamma_\ell(A)\ne 0$.

\subsubsection{\sc Example}\label{exx}
Consider $\Phi_\hbar=\hbar\mathrm{Cycl}(-zyx)$ where $\mathrm{Cycl}(-)$ denotes the cyclic word associated to a non-commutative monomial.
Then the algebra $\mathscr{U}_\hbar^\pol(\Phi_\hbar)$ has the following relations:
\begin{equation}\label{strange}
\begin{aligned}
\ &[x,y]=-\hbar yx\\
&[y,z]=-\hbar zy\\
&[z,x]=-\hbar xz
\end{aligned}
\end{equation}
Its specialization at $\hbar=1$ gives the associative algebra with relations:
\begin{align}
&xy=0\notag\\
&yz=0\notag\\
&zx=0\notag
\end{align}
The basis of this algebra as a $k$-vector space is formed by the monomials:
$$
x^{m_1}z^{m_2}y^{m_3}x^{m_4}z^{m_5}y^{m_6}\dots
$$
with all $m_i\ge 0$. It is clear that the graded components of this algebra are much bigger than the corresponding graded components for $S(x,y,z)$. That is, this algebra is not a PBW algebra.

It follows from Theorem \ref{main1}(i) that for $\Phi_\hbar=-\hbar zyx$ one should have
\begin{equation}\label{krull1}
I(\Phi_\hbar)=\bigcap_{\ell=0}^\infty \Gamma_\ell\mathscr{U}_\hbar(\Phi_\hbar)\ne 0
\end{equation}

Remind the Krull intersection theorem:
\begin{theorem}[Krull]
Let $A$ be a noetherian commutative ring, $\frak{q}$ an ideal of $A$, and $M$ a finitely generated $A$-module. Then $m\in \cap_n \frak{q}^nM$ if and only if there exists $d\in\frak{q}$ such that $dm=m$. In particular, if $\cap_n\frak{q}^nM\ne 0$, there exist $0\ne m\in M$ and $d\in\mathfrak{q}$ such that $dm=m$.
\end{theorem}
See e.g. [M, Theorem 8.9] for a proof.

For some non-commutative rings, the analogous result still holds, see [MC].

We can apply the result from loc.cit. to $A=\mathscr{U}_\hbar(\Phi_\hbar)$, for $\Phi_\hbar=-\hbar zyx$, and for $\frak{q}=\hbar\mathscr{U}_\hbar(\Phi_\hbar)$, $M=A$. It is instructive to find explicitly an element $m\in \mathscr{U}_\hbar(\Phi_\hbar)$ in the intersection $\bigcap_{\ell\ge 0}\Gamma_\ell(A)$. It turns out that such $m$ can be found {\it cubic} in $x,y,z$, with $d=\hbar\cdot 1$.

For this end, consider the equations \eqref{strange}.
We can multiply the first among them by $z$ from the left and from the right (and thus we get 2 equations), multiply the second one by $x$ from the left and from the right, and multiply the third one by $y$ from the left and from the right. Overall, we get 6 equations listed below.
\begin{align}
(A1)\ \ \ &zxy-zyx=-\hbar zyx\notag\\
(C1)\ \ \ &xyz-yxz=-\hbar yxz\notag\\
(C2)\ \ \ &xyz-xzy=-\hbar xzy\notag\\
(B1)\ \ \ &yzx-zyx=-\hbar zyx\notag\\
(B2)\ \ \ &yzx-yxz=-\hbar yxz\notag\\
(A2)\ \ \ &zxy-xzy=-\hbar xzy\notag
\end{align}
We group these 6 equations in three groups, labeled $A$,$B$, and $C$. Each group contains two equations, having equal first summands in the l.h.s.
Thus, the l.h.s. of $(A1)$ and $(A2)$ have equal to $zxy$ first summands.

Now we consider three new equations, obtained from these 6 equations as $(A1)-(A2)$, $(B1)-(B2)$, $(C1)-(C2)$. They are
\begin{align}
(A)\ \ \ &-zyx+xzy=\hbar(-zyx+xzy)\notag\\
(B)\ \ \ &-zyx+yxz=\hbar(-zyx+yxz)\notag\\
(C)\ \ \ &-yxz+xzy=\hbar(-yxz+xzy)\notag
\end{align}
(notice that there is 1 linear dependance between these three equations).

Each of these three lines has form $T=\hbar\cdot T$ for some cubic $T$. Then we have:
$$
T=\hbar\cdot T=\hbar^2\cdot T=\hbar^3\cdot T=\dots
$$
which shows that each such $T$ belongs to $\bigcap_{\ell\ge 0}\Gamma_\ell(\mathscr{U}_\hbar(\Phi_\hbar))$.

Notice that any such $T$ is killed under the completion homomorphism, as the equation $(\hbar-1)T=0$, under formal power series in $\hbar$, implies $T=0$, as $\hbar-1$ becomes invertible.

\subsection{\sc General case of quadratic algebras}\label{section33}
\comment
Let $\alpha$ be a Poisson bivector on $k^n$, where $k$ is a field.
If in coordinates
\begin{equation}\label{pq1}
\alpha=\sum_{i,j}\alpha_{ij}\partial_i\wedge\partial_j,\ \ \ \alpha_{ij}=-\alpha_{ji}
\end{equation}
the condition $\{\alpha,\alpha\}=0$ reads: for any $i,j,k$
\begin{equation}\label{pq2}
\sum_s\left(\alpha_{is}\partial_s\alpha_{jk}+\alpha_{js}\partial_s\alpha_{ki}+\alpha_{ks}\partial_s\alpha_{ij}\right)=0
\end{equation}
Suppose now that $\alpha$ is quadratic, that is
\begin{equation}\label{pq3}
\alpha_{ij}=\sum_{a,b}\alpha_{ij}^{ab}x_ax_b,\ \ \ \alpha_{ij}^{ab}\in k,\ \alpha_{ij}^{ab}=\alpha_{ij}^{ba}
\end{equation}
Then \eqref{pq2} can be rewritten as
\begin{equation}\label{pq4}
\Cycl_{i,j,k}\Symm_{a,b,c}\sum_s(\alpha_{is}^{ab}\alpha_{jk}^{sc})=0
\end{equation}
(We can replace $\Cycl_{i,j,k}$ by the full alternation $\Alt_{i,j,k}$, for more ``symmetric'' notations).

\begin{defn}\label{special}{\rm
We call a quadratic bivector $\alpha=\sum_{i,j,a,b}\alpha_{ij}^{ab}x_ax_b\partial_j\wedge\partial_j$ {\it special} if
\begin{equation}\label{pq4bis}
\Cycl_{i,j,k}\sum_s\alpha_{is}^{ab}\alpha_{jk}^{sc}=0 \ \ \forall\ a,b,c,i,j,k
\end{equation}
It is clear that \eqref{pq4bis} implies \eqref{pq4}, by symmetrization in $a,b,c$, that is any special quadratic bivector is Poisson.
}
\end{defn}

For $\alpha$ as in \eqref{pq3}, denote
\begin{equation}\label{pq5}
\varphi_{ij}=\frac12\hbar\alpha_{ij}^{ab}(x_a\otimes x_b+x_b\otimes x_a)\in T^2(k^n)
\end{equation}

We have:
\begin{theorem}
Let $\alpha$ be a special (see Definition {\it \ref{special}}) quadratic (Poisson) bivector. Then the non-commutative quadratic polynomials
$$
\varphi_{ij}=\hbar\sum_{a,b}\alpha_{ij}^{ab}(x_a\otimes x_b+x_b\otimes x_a)
$$
obey the assumption of Definition \ref{pbwd2}, for $A=T(x_1,\dots,x_n)[\hbar]/(x_i\otimes x_j-x_j\otimes x_i-\varphi_{ij})_{1\le i<j\le n}$.

Consequently, its specialization $A_a$ is a PBW algebra for almost all $a\in k$, in the sense of Definition \ref{pbwd1bis}.
\end{theorem}

\begin{proof}
We firstly prove that the algebra $A$ enjoys the property of Definition \ref{pbwd2}, then Theorem \eqref{main1}(ii) gives the PBW property for almost all specializations.

As usual, the value $d(\xi_{ij})$ is fixed by
\begin{equation}\label{pq6}
d(\xi_{ij})=x_i\otimes x_j-x_j\otimes x_i-\varphi_{ij}=x_i\otimes x_j-x_j\otimes x_i-\frac12\hbar\alpha_{ij}^{ab}(x_a\otimes x_b+x_b\otimes x_a)\in T^2(k^n)
\end{equation}
We now define
\begin{equation}\label{pq7}
d(\xi_{ijk})=\Cycl_{ijk}([x_i,\xi_{jk}])+\hbar\Cycl_{ijk}\sum_{a,b}\alpha_{jk}^{ab}(\xi_{ia}\otimes x_b+x_a\otimes \xi_{ib}+x_b\otimes \xi_{ia}+\xi_{ib}\otimes x_a)
\end{equation}
We claim, that $d^2(\xi_{ijk})=0$ for any $i,j,k$.

Denote the first (resp., second) summand of \eqref{pq6} by $d^{(1)}$ (resp., $d^{(2)}$). As well, denote the first (resp. second) summand of \eqref{pq7} by $A$ (resp., $B$).

We claim that
\begin{equation}\label{pq8}
d^{(2)}(A)=-d^{(1)}(B)
\end{equation}
and
\begin{equation}\label{pq9}
d^{(2)}(B)=0
\end{equation}
(In the last two equations, it is assumed that $d^{(1)}$ and $d^{(2)}$ are extended by the Leibniz rule, and that $d^{(1)}(x_i)=d^{(2)}(x_i)=0$ for any $i$). In the proofs below we will see that \eqref{pq8} holds for any quadratic Poisson bivector, whence for \eqref{pq9}
the assumption that $\alpha$ is a {\it special} quadratic Poisson bivector, is essentially used.

{\it Proof of \eqref{pq8}:}

One has:
\begin{equation}\label{pqq1}
d^{(2)}(A)=-\hbar\Cycl_{ijk}[x_i,\varphi_{jk}]=\hbar\Cycl_{ijk}[x_i,\sum_{a,b}\alpha_{jk}^{ab}(x_a\otimes x_b+x_b\otimes x_a)]
\end{equation}
We rewrite the r.h.s. of \eqref{pqq1} as
\begin{equation}\label{pqq2}
-\hbar\Cycl_{ijk}\sum_{ab}\alpha_{jk}^{ab}([x_i,x_a]\otimes x_b+x_a\otimes [x_i,x_b]+[x_i,x_b]\otimes x_a+x_b\otimes [x_i,x_a])
\end{equation}
On the other hand,
\begin{equation}\label{pqq3}
\begin{aligned}
\ &d^{(1)}(B)=\hbar d^{(1)}(\Cycl_{ijk}\sum_{a,b}\alpha_{jk}^{ab}(\xi_{ia}\otimes x_b+x_a\otimes \xi_{ib}+x_b\otimes \xi_{ia}+\xi_{ib}\otimes x_a))=\\
&\hbar \Cycl_{ijk}\sum_{a,b}([x_i,x_a]\otimes x_b+x_a\otimes [x_i,x_b]+x_b\otimes [x_i,x_a]+[x_i,x_b]\otimes x_a)
\end{aligned}
\end{equation}
which is equal, with another sign, to \eqref{pqq2}. In fact, the form of the second summand $B$ in \eqref{pq7} was hinted by this computation.

{\it Proof of \eqref{pq9}:}

We have:
\begin{equation}\label{pq10}
\begin{aligned}
\ &d^{(2)}(B)=\frac12\hbar^2\Cycl_{ijk}\sum_{a,b,c,d}\alpha_{jk}^{ab}\times\\
&\times((\alpha_{ia}^{cd}(x_c\otimes x_d+x_c\otimes x_d)\otimes x_b+\alpha_{ib}^{cd}x_a\otimes(x_c\otimes x_d+x_d\otimes x_c)+\\
&+\alpha_{ia}^{cd}x_b\otimes(x_c\otimes x_d+x_d\otimes x_c)+\alpha_{ib}^{cd}(x_c\otimes x_d+x_d\otimes x_c)\otimes x_a))
\end{aligned}
\end{equation}
One can rewrite \eqref{pq10} as
\begin{equation}\label{pq11}
d^{(2)}(B)=\hbar^2\Cycl_{ijk}\sum_{abcd}\alpha_{jk}^{ab}\alpha_{ia}^{cd}((x_c\otimes x_d+x_d\otimes x_c)\otimes x_b+
x_b\otimes (x_c\otimes x_d+x_d\otimes x_c))
\end{equation}
One can rewrite it even further:
\begin{equation}\label{pq12}
d^{(2)}(B)=\hbar^2\Cycl_{ijk}\sum_{b,c,d}\left(\sum_a (2\alpha_{jk}^{ab}\alpha_{ia}^{cd}+\alpha_{jk}^{ad}\alpha_{ia}^{bc}+\alpha_{jk}^{ac}\alpha_{ia}^{bd})x_c\otimes x_d\otimes x_b)\right)
\end{equation}

For a general quadratic Poisson bivector $\alpha$ this expression is not 0, as \eqref{pq4} is not applicable, because the expression
\begin{equation}\label{pq13}
\sum_a (2\alpha_{jk}^{ab}\alpha_{ia}^{cd}+\alpha_{jk}^{ad}\alpha_{ia}^{bc}+\alpha_{jk}^{ac}\alpha_{ia}^{bd})
\end{equation}
is not symmetric in $b,c,d$.

However, if the stronger than \eqref{pq4} equation \eqref{pq4bis} holds, the expression in \eqref{pq13} is 0, and $d^{(2)}(B)=0$.

Theorem is proven.
\end{proof}

\endcomment
Suppose we are given an associative algebra of type \eqref{eq1} with {\it quadratic} $\varphi_{ij}$. Here we find a criterium under which this algebra enjoys the property of Definition \ref{pbwd2}, which guarantees that it is a PBW algebra in the sense of Definition \ref{pbwd1bis}, for generic specialization of $\hbar=a\in k$.

Let
\begin{equation}\label{f1}
\varphi_{ij}=\hbar\sum_{a,b}\alpha_{ij}^{ab}x_a\otimes x_b
\end{equation}
where
\begin{equation}
\alpha_{ij}^{ab}=-\alpha_{ji}^{ab},  \text{   but   }\alpha_{ij}^{ab}\ne \alpha_{ij}^{ba} \text{   in general}
\end{equation}

Our main result here is
\begin{theorem}
In the notations as above, suppose that for $\forall i,j,k,b,c,d$:
\begin{equation}\label{section3f}
\Cycl_{i,j,k}\sum_s (\alpha_{jk}^{sb}\alpha_{is}^{cd}+ \alpha_{jk}^{cs}\alpha_{is}^{db})=0
\end{equation}
Then the associative algebra $A$ in \eqref{eq1} enjoys the PBW-like property of Definition \eqref{pbwd2} for the descending filtration $\{\Gamma_\ell\}$, and the genuine PBW property of Definition \ref{pbwd1bis} for generic specialization $\hbar=a$.
\end{theorem}
\begin{remark}{\rm
Let $\{\alpha_{ij}^{ab}\}$ satisfies \eqref{section3f}. Let $\beta=\beta_{ij}\partial_i\wedge\partial_j$, $\beta_{ij}=\sum_{a,b}(\alpha_{ij}^{ab}+\alpha_{ij}^{ba})x_ax_b$, be the corresponding bivector on $V=\{x_1,\dots,x_n\}$. Then the bivector $\beta$ is Poisson: in coordinates, for any $i,j,k,a,b,c$:
\begin{equation}\label{poisson}
\sum_s(\beta_{is}^{ab}\beta_{jk}^{sc}+\beta_{js}^{ab}\beta_{ki}^{sc}+\beta_{ks}^{ab}\beta_{ij}^{sc})=0
\end{equation}
Indeed, \eqref{poisson} is obtained from \eqref{section3f} by symmetrization in the upper indices. It is very natural, as the Poisson condition is equivalent to the flatness in degree 2 in $x_i$'s.
}
\end{remark}

\begin{proof}
We prove the first claim, that the corresponding algebra $A$ enjoys the property of Definition \ref{pbwd2}. The last claim follows then immediately from Theorem \ref{main1}(ii).

By Section \ref{concl} and Theorem \ref{lemma21bis}, we need to perturb the (-2)-component of the differential $d_\hbar^{(-2)}(\xi_{ijk})$, such that
\begin{equation}\label{f2}
d_\hbar^{(-1)}(\xi_{ij})=x_i\otimes x_j-x_j\otimes x_i-\varphi_{ij}
\end{equation}
where $\varphi_{ij}$ is as in \eqref{f1}. Then the only property to be checked is
\begin{equation}\label{f4}
d^{(-1)}_\hbar\circ d_\hbar^{(-2)}(\xi_{ijk})=0
\end{equation}
for any $i,j,k$.

We define
\begin{equation}\label{f3}
d_\hbar^{(-2)}(\xi_{ijk})=\Cycl_{i,j,k}[x_i,\xi_{jk}]+\hbar \Cycl_{i,j,k}\sum_{a,b}\alpha_{jk}^{ab}(\xi_{ia}\otimes x_b+x_a\otimes \xi_{ib})
\end{equation}

Then \eqref{f4} has the terms of order 0,1,2 in $\hbar$.
The cancelation of $\hbar^0$-terms is just the condition $d^2=0$ for the non-perturbed differential.
The cancelation in order $\hbar^1$ is straightforward and does not use \eqref{section3f}; in fact, the formula \eqref{f3} was designed especially to maintain this cancelation.

We need only to check \eqref{f4} in order $\hbar^2$ then.
We have:
\begin{equation}\label{f5}
d^2_\hbar(\xi_{ijk})=\hbar^2\Cycl_{ijk}\sum_{abcd}(\alpha_{jk}^{ab}\alpha_{ia}^{cd}x_c\otimes x_d\otimes x_b+\alpha_{jk}^{ab}\alpha_{ib}^{cd}x_a\otimes x_c \otimes x_d)
\end{equation}
To have valid \eqref{f4}, the coefficient at $x_a\otimes x_b\otimes x_c$ for any $a,b,c$ (and for any $i,j,k$) should vanish. It gives our condition \eqref{section3f}.

\end{proof}

\begin{remark}{\rm
For an algebra $A$ of type \eqref{eq1}, one has
\begin{equation}\label{re1}
[x_i,x_j]=\hbar\varphi_{ij}
\end{equation}
where by $x_i$'s and by $\varphi_{ij}$ is meant their image in the quotient-algebra $T(V)[\hbar]/\mathscr{I}$. The Jacobi identity
\begin{equation}\label{re2}
\Cycl_{i,j,k}[x_i,[x_j,x_k]]=0
\end{equation}
holds in any associative algebra. Using \eqref{re1} twice, it gives
\begin{equation}\label{re3}
\overline{\Cycl_{ijk}\sum_{a,b,c,d}(\alpha_{jk}^{ab}\alpha_{ia}^{cd}x_c\otimes x_d\otimes x_b+ \alpha_{jk}^{ab}\alpha_{ib}^{cd}x_a\otimes x_c\otimes x_d)}=0
\end{equation}
if
\begin{equation}\label{re4}
\varphi_{ij}=\hbar\sum_{a,b}\alpha_{ij}^{ab}x_a\otimes x_b
\end{equation}
where $\overline{\omega}$ for an $\omega\in T(V)[\hbar]$ stands for its image in the quotient-algebra $A=T(V)[\hbar]/\mathscr{I}$.

That is, \eqref{re3} holds for any $\{\varphi_{ij}\}$. Our condition \eqref{section3f} is exactly \eqref{re3} without the overline sign.

That is, it says that not only the image $\overline{\omega}$ of
\begin{equation}\label{re5}
\omega=\Cycl_{ijk}\sum_{a,b,c,d}(\alpha_{jk}^{ab}\alpha_{ia}^{cd}x_c\otimes x_d\otimes x_b+ \alpha_{jk}^{ab}\alpha_{ib}^{cd}x_a\otimes x_c\otimes x_d)
\end{equation}
vanishes in the quotient-algebra $T(V)[\hbar]/\mathscr{I}$ but, essentially stronger, $\omega$ itself is a zero element of $T(V)[\hbar]$.
}
\end{remark}

The Etingof-Ginzburg algebras do not obey, in general, the condition \eqref{section3f}. Seemingly, the case of dimension 3 is very special for examples of PBW algebras, due to some ``additional symmetries'' in dimension 3. When we are interested in PBW algebras defined over {\it polynomials} $k[\hbar]$, this symmetry makes possible the existence of some ``irregular'' examples.

\bigskip

\noindent {\sc Universiteit Antwerpen, Campus Middelheim, Wiskunde en Informatica, Gebouw G\\
Middelheimlaan 1, 2020 Antwerpen, Belgi\"{e}}

\bigskip

\noindent{{\it e-mail}: {\tt Boris.Shoikhet@ua.ac.be}}

\end{document}